\documentclass[12pt, a4paper]{amsart}
\usepackage{amssymb,latexsym}
\usepackage{epic}
\begin{document}




\title[Cone of Betti diagrams of bigraded artinian modules]
{The cone of Betti diagrams of 
bigraded artinian modules of codimension two}

\author{Mats Boij}
\address{Institusjonen for Matematik, KTH \\
 S-100 44 Stockholm \\
Sweden}


\author{Gunnar Fl{\o}ystad}
\address{Matematisk Institutt\\
         Johs. Brunsgt. 12\\
         5008 Bergen \\
         Norway}
\email{gunnar@mi.uib.no}

\date{\today}


\theoremstyle{plain}
\newtheorem{theorem}{Theorem}[section]
\newtheorem{corollary}[theorem]{Corollary}
\newtheorem*{main}{Main Theorem}
\newtheorem{lemma}[theorem]{Lemma}
\newtheorem{proposition}[theorem]{Proposition}
\newtheorem{conjecture}[theorem]{Conjecture}
\newtheorem*{theoremA}{Theorem}
\newtheorem*{theoremB}{Theorem B}

\theoremstyle{definition}
\newtheorem{definition}[theorem]{Definition}

\theoremstyle{remark}
\newtheorem{notation}[theorem]{Notation}
\newtheorem{remark}[theorem]{Remark}
\newtheorem{example}[theorem]{Example}
\newtheorem{claim}{Claim}


\newcommand{\psp}[1]{{{\bf P}^{#1}}}
\newcommand{\psr}[1]{{\bf P}(#1)}
\newcommand{\op}{{\mathcal O}}
\newcommand{\opw}{\op_{\psr{W}}}
\newcommand{\go}{\op}

\newcommand{\ini}[1]{\text{in}(#1)}
\newcommand{\gin}[1]{\text{gin}(#1)}
\newcommand{\kr}{{\Bbbk}}
\newcommand{\kk}{{\Bbbk}}
\newcommand{\pd}{\partial}
\newcommand{\vardel}{\partial}
\renewcommand{\tt}{{\bf t}}


\newcommand{\coh}{{{\text{{\rm coh}}}}}


\newcommand{\modv}[1]{{#1}\text{-{mod}}}
\newcommand{\modstab}[1]{{#1}-\underline{\text{mod}}}

\newcommand{\sut}{{}^{\tau}}
\newcommand{\sumit}{{}^{-\tau}}
\newcommand{\til}{\thicksim}

\newcommand{\totp}{\text{Tot}^{\prod}}
\newcommand{\dsum}{\bigoplus}
\newcommand{\dprod}{\prod}
\newcommand{\lsum}{\oplus}
\newcommand{\lprod}{\Pi}

\newcommand{\La}{{\Lambda}}
\newcommand{\lam}{{\lambda}}
\newcommand{\GL}{{GL}}

\newcommand{\sirstj}{\circledast}

\newcommand{\she}{\EuScript{S}\text{h}}
\newcommand{\cm}{\EuScript{CM}}
\newcommand{\cmd}{\EuScript{CM}^\dagger}
\newcommand{\cmri}{\EuScript{CM}^\circ}
\newcommand{\cler}{\EuScript{CL}}
\newcommand{\clerd}{\EuScript{CL}^\dagger}
\newcommand{\clerri}{\EuScript{CL}^\circ}
\newcommand{\gor}{\EuScript{G}}
\newcommand{\gF}{\mathcal{F}}
\newcommand{\gG}{\mathcal{G}}
\newcommand{\gM}{\mathcal{M}}
\newcommand{\gE}{\mathcal{E}}
\newcommand{\gD}{\mathcal{D}}
\newcommand{\gI}{\mathcal{I}}
\newcommand{\gP}{\mathcal{P}}
\newcommand{\gK}{\mathcal{K}}
\newcommand{\gL}{\mathcal{L}}
\newcommand{\gS}{\mathcal{S}}
\newcommand{\gC}{\mathcal{C}}
\newcommand{\gO}{\mathcal{O}}
\newcommand{\gJ}{\mathcal{J}}
\newcommand{\gU}{\mathcal{U}}
\newcommand{\mm}{\mathfrak{m}}

\newcommand{\dlim} {\varinjlim}
\newcommand{\ilim} {\varprojlim}

\newcommand{\CM}{\text{CM}}
\newcommand{\Mon}{\text{Mon}}


\newcommand{\Kom}{\text{Kom}}


\newcommand{\EH}{{\mathbf H}}
\newcommand{\res}{\text{res}}
\newcommand{\Hom}{\text{Hom}}
\newcommand{\inhom}{{\underline{\text{Hom}}}}
\newcommand{\Ext}{\text{Ext}}
\newcommand{\Tor}{\text{Tor}}
\newcommand{\ghom}{\mathcal{H}om}
\newcommand{\gext}{\mathcal{E}xt}
\newcommand{\id}{\text{{id}}}
\newcommand{\im}{\text{im}\,}
\newcommand{\codim} {\text{codim}\,}
\newcommand{\resol}{\text{resol}\,}
\newcommand{\rank}{\text{rank}\,}
\newcommand{\lpd}{\text{lpd}\,}
\newcommand{\coker}{\text{coker}\,}
\newcommand{\supp}{\text{supp}\,}
\newcommand{\Ad}{A_\cdot}
\newcommand{\Bd}{B_\cdot}
\newcommand{\Fd}{F_\cdot}
\newcommand{\Gd}{G_\cdot}


\newcommand{\sus}{\subseteq}
\newcommand{\sups}{\supseteq}
\newcommand{\pil}{\rightarrow}
\newcommand{\vpil}{\leftarrow}
\newcommand{\rpil}{\leftarrow}
\newcommand{\lpil}{\longrightarrow}
\newcommand{\inpil}{\hookrightarrow}
\newcommand{\pils}{\twoheadrightarrow}
\newcommand{\projpil}{\dashrightarrow}
\newcommand{\dotpil}{\dashrightarrow}
\newcommand{\adj}[2]{\overset{#1}{\underset{#2}{\rightleftarrows}}}
\newcommand{\mto}[1]{\stackrel{#1}\longrightarrow}
\newcommand{\vmto}[1]{\overset{\tiny{#1}}{\longleftarrow}}
\newcommand{\mtoelm}[1]{\stackrel{#1}\mapsto}

\newcommand{\eqv}{\Leftrightarrow}
\newcommand{\impl}{\Rightarrow}

\newcommand{\iso}{\cong}
\newcommand{\te}{\otimes}
\newcommand{\into}[1]{\hookrightarrow{#1}}
\newcommand{\ekv}{\Leftrightarrow}
\newcommand{\equi}{\simeq}
\newcommand{\isopil}{\overset{\cong}{\lpil}}
\newcommand{\equipil}{\overset{\equi}{\lpil}}
\newcommand{\ispil}{\isopil}
\newcommand{\vvi}{\langle}
\newcommand{\hvi}{\rangle}
\newcommand{\susneq}{\subsetneq}
\newcommand{\sgn}{\text{sign}}


\newcommand{\xd}{\check{x}}
\newcommand{\ortog}{\bot}
\newcommand{\tL}{\tilde{L}}
\newcommand{\tM}{\tilde{M}}
\newcommand{\tH}{\tilde{H}}
\newcommand{\tvH}{\widetilde{H}}
\newcommand{\tvh}{\widetilde{h}}
\newcommand{\tV}{\tilde{V}}
\newcommand{\tS}{\tilde{S}}
\newcommand{\tT}{\tilde{T}}
\newcommand{\tR}{\tilde{R}}
\newcommand{\tf}{\tilde{f}}
\newcommand{\ts}{\tilde{s}}
\newcommand{\tp}{\tilde{p}}
\newcommand{\tr}{\tilde{r}}
\newcommand{\tfst}{\tilde{f}_*}
\newcommand{\empt}{\emptyset}
\newcommand{\bfa}{{\bf a}}
\newcommand{\bfb}{{\bf b}}
\newcommand{\bfd}{{\bf d}}
\newcommand{\bfe}{{\bf e}}
\newcommand{\bfp}{{\bf p}}
\newcommand{\bfc}{{\bf c}}
\newcommand{\bfl}{{\bf t}}
\newcommand{\la}{\lambda}
\newcommand{\bfen}{{\mathbf 1}}
\newcommand{\ep}{\epsilon}
\newcommand{\en}{p}
\newcommand{\tu}{q}
\newcommand{\fel}{m}

\newcommand{\ome}{\omega_E}

\newcommand{\bevis}{{\bf Proof. }}
\newcommand{\demofin}{\qed \vskip 3.5mm}
\newcommand{\nyp}[1]{\noindent {\bf (#1)}}
\newcommand{\demo}{{\it Proof. }}
\newcommand{\demodone}{\demofin}
\newcommand{\parg}{{\vskip 2mm \addtocounter{theorem}{1}  
                   \noindent {\bf \thetheorem .} \hskip 1.5mm }}

\newcommand{\red}{{\text{red}}}
\newcommand{\lcm}{{\text{lcm}}}


\newcommand{\dl}{\Delta}
\newcommand{\cdel}{{C\Delta}}
\newcommand{\cdelp}{{C\Delta^{\prime}}}
\newcommand{\dlst}{\Delta^*}
\newcommand{\Sdl}{{\mathcal S}_{\dl}}
\newcommand{\lk}{\text{lk}}
\newcommand{\lkd}{\lk_\Delta}
\newcommand{\lkp}[2]{\lk_{#1} {#2}}
\newcommand{\del}{\Delta}
\newcommand{\delr}{\Delta_{-R}}
\newcommand{\dd}{{\dim \del}}

\renewcommand{\aa}{{\bf a}}
\newcommand{\bb}{{\bf b}}
\newcommand{\cc}{{\bf c}}
\newcommand{\xx}{{\bf x}}
\newcommand{\yy}{{\bf y}}
\newcommand{\zz}{{\bf z}}
\newcommand{\mv}{{\xx^{\aa_v}}}
\newcommand{\mF}{{\xx^{\aa_F}}}

\newcommand{\pnm}{{\bf P}^{n-1}}
\newcommand{\opnm}{{\go_{\pnm}}}
\newcommand{\ompnm}{\omega_{\pnm}}

\newcommand{\pn}{{\bf P}^n}
\newcommand{\hele}{{\mathbb Z}}
\newcommand{\nat}{{\mathbb N}}
\newcommand{\rasj}{{\mathbb Q}}

\newcommand{\dt}{{\displaystyle \cdot}}
\newcommand{\st}{\hskip 0.5mm {}^{\rule{0.4pt}{1.5mm}}}              
\newcommand{\disk}{\scriptscriptstyle{\bullet}}

\newcommand{\cF}{F_\dt}
\newcommand{\pol}{f}

\newcommand{\disc}{\circle*{5}}

\def\CC{{\mathbb C}}
\def\GG{{\mathbb G}}
\def\ZZ{{\mathbb Z}}
\def\NN{{\mathbb N}}
\def\RR{{\mathbb R}}
\def\OO{{\mathbb O}}
\def\QQ{{\mathbb Q}}
\def\VV{{\mathbb V}}
\def\PP{{\mathbb P}}
\def\EE{{\mathbb E}}
\def\FF{{\mathbb F}}
\def\AA{{\mathbb A}}

\begin{abstract}
We describe the positive cone generated by bigraded Betti diagrams
of artinian modules of codimension two, whose resolutions become 
pure of a given type when taking total degrees. 
If the differences of these total degrees, $p$ and $q$,  are relatively
prime, the extremal rays are parametrised by order
ideals in ${\mathbb N}^2$ contained in the region
$px + qy < (p-1)(q-1)$. We also consider some examples concerning
artinian modules of codimension three.
\end{abstract}

\maketitle

\section*{Introduction}

In \cite{EFW}, D.Eisenbud, J.Weyman, and the second author gave
for every sequence of integers $\bfd : d_0 < d_1 < \cdots < d_n$
 a construction
of pure resolutions of graded artinian modules over a polynomial ring
$S = \kk[x_1, \ldots, x_n]$ (char $\kk = 0$)
\[ S(-d_0)^{\beta_0} \vpil S(-d_1)^{\beta_1} 
\vpil \ldots \vpil S(-d_n)^{\beta_n}. \]
Moreover these resolutions were $\GL(n)$-equivariant, and so
in particular invariant for the diagonal matrices and hence $\hele^n$-graded.

In the case when $S = \kk[x_1,x_2]$, the first author and J.S\"oderberg
in \cite[Remark 3.2]{BS} gave a different construction of pure resolutions of artinian
bigraded modules. It had a bigraded Betti diagram distinct from that of the
equivariant resolution.

\begin{example}
Suppose
$d_1 - d_0 = 2$ and $d_2 - d_1 = 3$. 
The equivariant resolution has the following form where we have written
the bidegrees of the generators below the terms.
\begin{equation} \label{IntroLigEkvi}
\underset{\scriptsize{\begin{matrix} (2,0) \\ (1,1) \\ (0,2) \end{matrix}}} {S^3} \vpil 
\underset{\scriptsize{\begin{matrix} (4,0) \\ (3,1) \\ (2,2) \\ (1,3) \\ (0,4) \end{matrix}}}
{S^5} \vpil
\underset{\scriptsize{\begin{matrix} (4,3) \\ (3,4) \end{matrix}}} {S^2}.
\end{equation}
Let $\beta_1$ be its bigraded Betti table.
The resolution in \cite{BS} is of a quotient of a pair of monomial ideals. 
For the type above the resolution has the following
bidegrees.
\begin{equation} \label{IntroLigBS}
\underset{\scriptsize{\begin{matrix} (4,0) \\ (2,2) \\ (0,4) \end{matrix}}} 
{S^3} \vpil 
\underset{\scriptsize{\begin{matrix} (6,0) \\ (4,2) \\ (3,3) \\ (2,4) \\ (0,6) 
\end{matrix}}} {S^5} \vpil
\underset{\scriptsize{\begin{matrix} (6,3) \\ (3,6) \end{matrix}}} {S^2}.
\end{equation}
Denote by $\beta_2$ be its Betti diagram.
\end{example}

This indicated that there may be many types of multigraded
Betti diagrams of ${\hele}^n$-graded artinian modules of codimension $n$
whose resolutions become pure of a given type when taking total degrees.
In \cite{Fl} the second author showed that the multigraded Betti
diagram of the equivariant resolution has a fundamental position. This
diagram and its twists with $\bfa \in \hele^n$ form a basis for the linear space
generated by  multigraded Betti diagrams of artinian $\hele^n$-graded modules whose 
resolutions become pure of the given type when taking total degrees. Even more natural
it is to describe the positive cone generated by the multigraded Betti diagrams.

In this paper we to this in the case when $S = \kk[x_1, x_2]$. Let 
$e_1 = d_1 - d_0$ and $e_2 = d_2 - d_1$. 
We describe all the extremal rays of the positive cone $P(e_1,e_2)$ generated
by bigraded Betti diagrams of artinian bigraded modules of codimension two whose
resolutions become pure when taking total degrees, and where the differences
of these total degrees are $e_1$ and $e_2$. In the example above the two
resolutions, or rather their Betti diagrams, are essentially the full story in 
the sense that the extremal rays in 
$P(2,3)$ are exactly the rays generated by
$\beta_1(\bfa)$ and $\beta_2(\bfa)$ for $\bfa \in \hele^2$. 
To explain the general situation assume here for simplicity that
$e_1$ and $e_2$ are relatively prime. 
Let $R(e_1,e_2)$ be the integer coordinate points 
in the region of the first quadrant of the coordinate plane bounded by the line
$e_1 x + e_2 y < (e_1 - 1)(e_2 - 1)$. There is a partial order on 
$\nat^2$ given by $(a_1,a_2) \leq (b_1, b_2)$ if $a_1 \leq a_2$ and
$b_1 \leq b_2$, and the region $R(e_1,e_2)$ inherits this. An order
ideal in $R(e_1,e_2)$ corresponds to a partition $\la$. 
We give a construction which to each  partition $\la$ in $R(e_1,e_2)$
associates a bigraded resolution 
\[ S^{e_2} \vpil S^{e_2 + e_1} \vpil S^{e_1}. \]
Let $\beta_\la$ be the bigraded Betti diagram of this complex.
The following is our main result
in the case that $e_1$ and $e_2$ are relatively prime.

\begin{theoremA}
The extremal rays in the cone $P(e_1,e_2)$ are the $\beta_\la(\bfa)$
where $\bfa$ varies over $\hele^2$ and $\la$ ranges over partitions
contained in the region $R(e_1,e_2)$. 
\end{theoremA}

The general case is formulated in Theorems \ref{PosconeTheMain}
and \ref{EksisTheMain}.
In the region $R(e_1,e_2)$ there are two distinguished partitions, the maximal one
and the empty one. It turns out that the maximal one corresponds
to the equivariant complex and the empty one corresponds to the
bigraded resolution of a quotient of monomial ideals
constructed in \cite{BS}.

The organisation of the paper is as follows. Section 1 
contains preliminaries. First we give the multigraded Herzog-K\"uhl equations
which give strong restrictions on Betti diagrams of multigraded
artinian modules. We recall the equivariant resolution, and the result
of \cite{Fl} that its twists generate the linear space of multigraded Betti
diagrams of artinian ${\hele}^n$-graded modules of codimension $n$
whose resolution becomes pure when taking total degrees.
This give us a very simple alternative description of  the positive cone $P(e_1,e_2)$.
This is used in Section 2 where we show that the
extremal rays of the positive cone $P(e_1,e_2)$ are generated by the 
Betti diagrams $\beta_\la(\bfa)$ for $\bfa \in \hele^2$, provided
these diagrams really come from resolutions. And that such resolutions
really exist is established in Section 3. 
In Section 4 we briefly discuss the positive cone in the case of three 
variables, providing an example.

\section{Preliminaries}

Let $S = \kr[x_1, \ldots, x_n]$ be the polynomial ring over a field $\kr$.
We shall study $\hele^ n$-graded free resolutions of artinian 
$\hele^ n$-graded $S$-modules
\[ F_0 \vpil F_1 \vpil \cdots \vpil F_n.  \]
 For a multidegree $\bfa = (a_1, a_2, \ldots, a_n)$ in $\hele^ n$
let $|\bfa| = \sum a_i$ be its total degree. We shall be interested
in the case that these resolutions become pure resolutions if we make
them singly graded by taking total degrees. That is there is a
sequence $d_0 < d_1 < \cdots < d_n$ such that 
\[ F_i = \oplus_{|\bfa| = d_i} S(- \bfa)^ {\beta_{i, \bfa}}. \]

\subsection {Betti diagrams and the multigraded Herzog-K\"uhl 
equations}
The {\it multigraded Betti diagram} of such a resolution is the element
\[ \{\beta_{i,\bfa}\}_{\scriptsize \underset{}
{\begin{matrix} i=0, \ldots, n \\ \bfa \in \hele^n \end{matrix}  }}
\in \oplus_{\hele^n} \nat^{n+1}. \] 

\medskip
A way of representing
a multigraded Betti table which will be very convenient for us
is to represent $\beta = \{ \beta_{i, \bfa} \}$
where $i = 0,\ldots,n$ and $\bfa \in \hele^{n}$ by Laurent polynomials
\[ B_i(t) = \sum_{\bfa \in \hele^n} \beta_{i,\bfa} \cdot t^\bfa. \]
We thus get an $(n+1)$-tuple of Laurent polynomials
\[ B = (B_0, B_1, \ldots, B_n). \]
Also the module $\oplus_\bfa S(-\bfa)^{\beta_{i,\bfa}}$ may be conveniently denoted
as $S.B_i$.

\medskip
Let $e_i = d_i - d_{i-1}$, so we get the differences $\bfe  =
(e_1, \ldots, e_n)$. 
Now let $L(\bfe)$ be the linear subspace of $\oplus_{\bfa \in \hele^n} 
{\mathbb Q}^{n+1}$ generated by multigraded Betti diagrams of
$\hele^n$-graded artinian $S$-modules whose resolutions become pure when taking total
degrees, and where the difference sequence of these total degrees is $\bfe$. 
Similarly let $P(\bfe)$ be the positive cone in $\oplus_{\bfa \in \hele^n} 
{\mathbb Q}^{n+1}$ generated by such Betti diagrams.

\medskip
There are some natural restrictions on $L(\bfe)$ coming from the 
multigraded Herzog-K\"uhl equations.
If the resolution resolves the artinian module $M$, 
the multigraded Hilbert series of $M$ is the polynomial
\[ h_M(t) = \frac{\sum_{i,\bfa} (-1)^i \beta_{i, \bfa} \cdot t^ {\bfa}} 
{\Pi_{k=1}^ n (1-t_i)}, \] 
which gives
\begin{equation} \label{SetLigBeta} 
\sum_{i,\bfa} (-1)^ i \beta_{i, \bfa} t^ {\bfa} = h_M(t) \cdot \Pi_{k=1}^n(1-t_i). 
\end{equation}
For each multigrade $\bfa \in \hele^ {n}$ and integer $k = 1, \ldots, n$, let 
the projection $\pi_k(\bfa)$ be $(a_1,\ldots, \hat{a}_k, \ldots, a_n)$, 
the $n-1$-tuple where we omit $a_k$. 

Now we have the multigraded analogs of the Herzog-K\"uhl (HK) equations. 
We obtain these by setting $t_k = 1$ in (\ref{SetLigBeta}) for each $k$. 
This gives for every $\hat{\bfa}$ in $\hele^{n-1}$ and $k = 1, \ldots, n$ 
an equation
\begin{equation} \label{SetLigHK}
\sum_{i,\pi_k(\bfa)= \hat{\bfa}} (-1)^i \beta_{i, \bfa} = 0.
\end{equation}

Let $L^\prime(\bfe)$ be the linear space of elements in $\oplus_{\bfa \in \hele^{n}} 
{\mathbb Q}^{n+1}$ which fulfil the multigraded HK-equations above, 
and which become pure diagrams when taking total degrees 
with the  difference sequence of these total degrees equal to  $\bfe$. 
Also let $P^\prime(\bfe)$ be
the cone in $L^\prime(\bfe)$ consisting of the elements with nonnegative coordinates.
There are natural injections $L(\bfe) \pil L^\prime(\bfe)$ and 
$P(\bfe) \pil P^\prime(\bfe)$. 
In \cite{Fl} the second author showed that the first injection is an isomorphism
and moreover gave an explicit basis for $L(\bfe)$ which we now describe.

\subsection{The equivariant resolution}
In \cite{EFW} the second author together with D.Eisenbud and J.Weyman
constructed $\GL(n)$-equivariant pure resolutions of artinian modules. 
For a partition $\lambda = (\la_1, \ldots,
\la_n)$ let $S_\lambda$ be the associated Schur module, it is an
irreducible representation of $\GL(n)$ (see for instance \cite{FuH}).
The action of the diagonal matrices in $\GL(n)$ gives a decomposition of 
$S_\lambda$ as a $\hele^n$-graded vector space. The basis elements are
given by semi-standard Young tableau of shape $\la$ with entries from 
$1,2, \ldots, n$. All the nonzero graded pieces in this decomposition
 have total degree $|\la| = \sum_{i=1}^n \la_i$. 
The free module $S \te_k S_\la$ then becomes a free multigraded module
where the generators all have total degree $|\la|$.

Now given the difference vector $\bfe$, let 
\[ \la_i = \Sigma_{j = i+1}^n e_j - 1 \] and define a sequence of 
partitions for $i=0, \ldots, n$ by
\begin{equation*}
\alpha(\bfe,i) = (\la_1 + e_1, \la_2 + e_2, \ldots, \la_i + e_i, \la_{i+1}, 
\ldots, \la_n).
\end{equation*}
The construction in  \cite{EFW} then gives a $\GL(n)$-equivariant resolution
\begin{equation} \label{SetLigEe}
E(\bfe) :  S \te_k S_{\alpha(\bfe,0)} \vpil
S \te_k S_{\alpha(\bfe,1)} \vpil \cdots \vpil S \te_k S_{\alpha(\bfe,n)}
\end{equation}
of an artinian $S$-module. 

In the case of two variables $S = \kk[x_1, x_2]$ the resolution takes 
the form
\begin{equation} \label{SetLigEqui2} 
E(e_1,e_2): S \te_k S_{e_2 - 1,0} \vpil S \te_k S_{e_1 + e_2 -1, 0} 
\vpil S \te_k S_{e_1 + e_2 - 1, e_2}. 
\end{equation}


\subsection{The linear space of Betti diagrams of multigraded 
artinian modules}

For a multigraded Betti diagram 
$\beta = \{ \beta_{i, \bfa} \}$ and a multidegree $\bfl$ in $\hele^ {n}$, 
we get the twisted Betti diagram $\beta(-\bfl)$ which in homological degree
$i$ and multidegree $\bfa$ is given by $\beta_{i,\bfa-\bfl}$. If $\cF$
is a resolution with Betti diagram $\beta$, then $\cF(-\bfl)$ is
a resolution with Betti diagram $\beta(-\bfl)$. 

Also let $F_r : S \pil S$ be the map sending $x_i \mapsto x_i^r$. 
Denote by $S^{(r)}$ the ring $S$ with the $S$-module structure
given by $F_r$.  Given any complex $\cF$ we may tensor it with 
$- \te_S S^{(r)}$ and get a complex we denote by $\cF^{(r)}$. 
Note that if $\cF$ is pure with degrees $\bfd$, then $\cF^{(r)}$ is
pure with degrees $r \cdot \bfd$. 

In \cite{Fl} we showed the following. 
\begin{theorem} \label{LinbettiTheMain}
Let $m = \gcd(e_1, \ldots, e_n)$ and let $\bfe = m \cdot \bfe^\prime$. 
The space $L(\bfe)$ is equal to the space $L^\prime(\bfe)$ of diagrams
fulfilling the HK-equations, and the $\beta_{E(\bfe^\prime)^{(m)}}(\bfa)$ 
where $\bfa$ varies over $\hele^n$, form a basis for $L(\bfe)$.

Moreover if $E^\prime$ is another resolution such that the $\beta_{E^\prime}(\bfa)$
form a basis, then
$\beta_{E^\prime}$ is an integer multiple of $\beta_{E(\bfe^\prime)^{(m)}}(\bfa)$
for some $\bfa$.
\end{theorem}

This may also be formulated in terms of the associated
$(n+1)$-tuple of Betti polynomials.

\begin{corollary} \label{SetCorLinpoly}
Let $s = (s_0, \ldots, s_n)$ be the $(n+1)$-tuple of Betti
polynomials of $E(\bfe^\prime)^{(m)}$. If $B = (B_0, \ldots, B_n)$ is any
$(n+1)$-tuple of Betti polynomials of an artinian $\hele^n$-graded 
module whose resolution becomes pure when taking total degrees and with difference
vector $\bfe$ of the total degrees, then $B = p \cdot s$ for some homogeneous
Laurent polynomial $p$.
\end{corollary}

\subsection{The linear space in the case of two variables}
\label{SetSubsecLinto}

Now assume $S = \kk[x_1, x_2]$.
Let $\xi_d(t,u) = t^{d-1} + t^{d-2}u + \cdots + u^{d-1}$ be the cyclotomic
polynomial. The first and last
Betti polynomials of the equivariant resolution (\ref{SetLigEqui2}) are then 
respectively
\[ \xi_{e_2}(t,u), \quad (tu)^{e_2} \xi_{e_1}(t,u) \]
and the middle Betti polynomial is 
\begin{equation} \label{SetLigXi} 
\xi_{e_1 + e_2} =  t^{e_2}\xi_{e_1}(t,u) + u^{e_1}\xi_{e_2}(t,u) = 
u^{e_2}\xi_{e_1}(t,u) + t^{e_1} \xi_{e_2}(t,u).
\end{equation}
By Corollary \ref{SetCorLinpoly} 
the space $L(e_1, e_2)$ may now be described as follows.

\begin{lemma} \label{SetLemB}
 Let $e_1 = \fel\tu$ and $e_2 = \fel\en$ where $m$ is the 
greatest common divisor of $e_1$ and $e_2$. 
A triple of homogeneous Laurent polynomials $B_0, B_1, B_2$ 
whose degrees have $e_1$ and $e_2$ as differences, is in $L(e_1, e_2)$
if and only if the following two equations hold:
\begin{eqnarray} \label{SetLigB02}
B_2(t,u) \cdot \xi_\en(t^\fel, u^\fel) &=& (tu)^{\fel\en} 
B_0(t,u) \cdot 
\xi_{\tu}(t^\fel, u^\fel), \\ \label{SetLigB03}
B_1(t,u) &=& u^{-\en \fel} B_2(t,u) + u^{\tu \fel} B_0(t,u) \\ \notag
                     &=& t^{-\en \fel} B_2(t,u) + t^{\tu \fel} B_0(t,u) .
\end{eqnarray}
\end{lemma}

\begin{proof}
By Corollary \ref{SetCorLinpoly} we have
\begin{equation} \notag (B_0, B_1, B_2) = f(t,u) \cdot 
(\xi_\tu(t^m, u^m), \xi_{\en + \tu}(t^m, u^m), (tu)^{\en + \tu}
\xi_{\tu}(t^m, u^m)).
\end{equation}
This gives (\ref{SetLigB02}). Also (\ref{SetLigB03}) follows by
(\ref{SetLigXi}).
Conversely, if (\ref{SetLigB02}) and (\ref{SetLigB03}) hold, we may deduce
that the equation above holds, so $(B_0,B_1,B_2)$ is in $L(e_1, e_2)$.
\end{proof}
 
For a homogeneous Laurent polynomial $f(t,u)$ denote by $f^{dh}(t)$ its
dehomogenisation with respect to $u$. If we now dehomogenise equation
(\ref{SetLigB02}) we get an equation
\[ B_2^{dh}/t^{\en \fel} \cdot \xi_\en(t^\fel) = B_0^{dh} \cdot \xi_\tu(t^\fel). \]
Each of the first factors are uniquely determined by the other, and if the triple
comes from an actual complex, the coefficients are non-negative.

With some abuse of notation we also identify the cone 
$P^\prime = P^\prime(e_1, e_2)$ with the 
positive cone of pairs of Laurent polynomials
$(A(t),B(t))$ in one variable $t$ and with non-negative coefficients, such that 
\[ B(t) \xi_\en(t^\fel) = A(t) \xi_\tu(t^\fel). \]
We shall in the next section describe the cone $P^\prime$ completely. 
Recall that we have an injective map $P(e_1,e_2) \pil P^\prime(e_1,e_2)$. 
In Section \ref{EksisSec} we show that this map is an isomorphism.

\section{The positive cone of bigraded Betti diagrams}
\label{PosSec}
In this section we describe completely the positive cone $P^\prime(e_1,e_2)$
of diagrams fulfilling the HK-equations (\ref{SetLigHK}).
We shall show that there is a finite number
of diagrams $\beta_\lambda$ parametrised by certain partitions $\la$ 
such that extremal rays in the positive cone
are the one-dimensional rays generated by $\beta_\la(\bfa)$
for $\bfa \in \hele$.

\medskip
{\noindent \it Note.} In the following we let $e_1 = \fel\tu$ and 
$e_2 = \fel\en$ where $\fel$ is the greatest common divisor of $e_1$ and $e_2$.

\subsection{Partitions}

Let ${\mathbb N}^2$ have the partial ordering where $(a_1,a_1) \leq (b_1,b_2)$
if $a_1 \leq b_1$ and $a_2 \leq b_2$.
An order ideal $T$ in ${\mathbb N}^2$ (a set closed under taking smaller elements)
gives rise to two partitions. The first is given by
\[ \lambda_j = 1 + \max \{ i \,| \, (i,j) \in T \}, \,\, j \geq 0. \]
The second is the dual partition 
\[ \mu_i = 1 + \max \{ j \, | \, (i,j) \in T \}, \,\, i \geq 0. \]
(If for a given $j$ no $(i,j)$ is in $T$, we set $\lambda_j =0$ and
correspondingly for $\mu_i$.)
Note that $\lambda$ and $\mu$ are dual partitions. So
$\mu_i$ is the cardinality of $\{ j\,| \, \lambda_j > i \}$.

\medskip
We shall be interested in order ideals $T$ which are contained in 
the region $R(\en,\tu)$ in the first quadrant bounded by the following strict inequality
\begin{equation} \notag
\en x+\tu y < (\en-1)(\tu-1). 
\end{equation}

\begin{lemma}
Let the order ideal  $T$ correspond to the partition $\lambda$.
Then $T$ is contained in the region above if and only if 
every $a\tu-\en \lambda_{\en-1-a}$ is nonnegative for $0 \leq a < \en$. 
Correspondingly for the dual partition $\mu$.
\end{lemma}

\begin{proof}
First note that $a\tu - \en \lambda_{\en-1-a} \geq 0$ if and only if
\[(\en-1-a)\tu + (\lambda_{\en-1-a}-1)\en \leq \en \tu-\en-\tu. \]
Assume $0 \leq a < p$. If $T$ is contained in $R(p,q)$ then if 
$\lambda_{p-1-a} \geq 1$ it fulfils the second equation above and therefore the first.
If $\lambda_{p-1-1} = 0$ the first equation is also fulfilled.
Suppose now that $T$ fulfils the first equation. Then when $\lambda_{p-1-a} \geq1$
the point $(p-1-a, \lambda_{p-1-a} - 1)$ is in $R(p,q)$, so $T$ is contained 
in $R(p,q)$. 
\end{proof}

The following easy lemma will be useful.

\begin{lemma}\label{PosLemHopp} Let $P(t) = \sum c_i t^i$ be a polynomial with 
positive coefficients. Write $P(t) \xi_d(t) =  \sum_{j \in \hele} \alpha_jt^j$.
Then $\alpha_j - \alpha_{j-1} = c_j - c_{j-d}$.
\end{lemma}

\begin{proof} This is clear from $\alpha_j = \sum_{i = j-d+1}^{j} c_i$.
\end{proof}

The following result will essentially describe the extremal rays.

\begin{proposition} Suppose $\en$ and $\tu$ are relatively prime and
let $T$ be an order ideal in $R(\en,\tu)$.
Write  
\[ A_T(t) = \sum_{a=0}^{\en-1} t^{a\tu-\en\lambda_{\en-1-a}},
\quad B_T(t) = \sum_{a=0}^{\tu-1} t^{a\en-\tu\mu_{\tu-1-a}}. \]
Then 
\[ A_T(t) \xi_\tu(t) = B_T(t) \xi_\en(t). \]
\end{proposition}

\begin{proof}
Note that since $\en$ and $\tu$ are relatively prime, 
the coefficient of each power $t^j$
in $A_T$ or $B_T$ is $0$ or $1$.  
Writing $\sum \alpha_j t^j$ for the product $A_T(t) \xi_\tu(t)$ we see that 
when $\alpha_j > \alpha_{j-1}$ we have $\alpha_j = \alpha_{j-1} + 1$.
We shall show that the indices $j$ for which this happens are exactly when $j = 0$ or
$j = \en\tu-\en-\tu-\tu u-\en v$ where $(u,v)$ is a maximal element in 
the poset $T$,
i.e. $(u,v)$ is in $T$, but neither $(u+1,v)$ nor $(u,v+1)$ is in $T$.
Since the analog holds for the product $B_T(t) \xi_\en(t)$, these products
must increase exactly at the same indices. An analog argument also show
that they decrease at exactly the same indices, namely $\alpha_j < \alpha_{j-1}$ 
iff $j = \en\tu-\tu u -\en v$ where $(u,v)$ is not in $T$ but
$(u-1,v)$ and $(u,v-1)$ are either in $T$ or have $-1$ as a coordinate. 
Hence the products are equal.

Now $\alpha_j > \alpha_{j-1}$ when $j = a\tu- \en\lambda_{\en-1-a}$ 
for some $a$
but $(a-1)\tu - \en \lambda_{\en-1-a}$ is not a power in $A(t)$. 
Thus either $a = 0$ or 
$\lambda_{\en-a}
< \lambda_{\en-1-a}$. But this means that $j = 0$ or  
$(u,v) = (\lambda_{\en-1-a} - 1,\en-1-a)$
is a maximal element in $T$.
We easily compute that
\[ j = a\tu-\en \lambda_{\en-1-a}  = \en\tu-\en-\tu-\tu u-\en v.\]
\end{proof}

\begin{remark} \label{PosRemMin}
The empty poset $T = \emptyset$ corresponds to the polynomials
\begin{eqnarray} \notag A_{\emptyset}(t) & = & \xi_\en(t^\tu) \\
 \notag  B_{\emptyset}(t) & = & \xi_\tu(t^\en).
\end{eqnarray}
Via the correspondence at the end of Subsection \ref{SetSubsecLinto}
these corresponds to the 
Betti diagram of a resolution of an artin
module. This is the module described in \cite[Remark 3.2]{BS} 
which is the quotient $I/J$ of 
two monomial ideals in $k[x,y]$: the ideal $I = (x^{(\en-1)\tu}, x^{(\en-2)\tu}y^\tu, 
\ldots, y^{(\en-1)\tu})$
and the ideal $J = (x^{\en\tu}, x^{\en(\tu-1)}y^\en, 
\ldots, y^{\en\tu})$
\end{remark}

\begin{remark}\label{PosRemMax}
There is also a maximal order ideal $\hat T$ in the region $R(\en,\tu)$ and this corresponds
to the polynomials
\begin{eqnarray} 
\notag   A_{\hat T}(t) & = & \xi_\en(t) \\
\notag   B_{\hat T}(t) & = & \xi_\tu(t)
\end{eqnarray} 
which again via the correspondence at the end of Subsection \ref{SetSubsecLinto}
corresponds to the Betti diagrams of the $GL(2)$-equivariant resolutions
$E(p,q)$ constructed in \cite{EFW}.
\end{remark}

\subsection{Decomposing}
Now any polynomial $A(t)$ may be written
\[ A(t) = \sum_{a=0}^{\en-1} \sum_{b \in \hele} \alpha_{a,b} t^{a\tu-b\en}. \]
For each $a$ let $\lambda_{\en-1-a}$ be the maximum of the set 
$\{ b \, | \, \alpha_{a,b} \neq 0 \}$. We may then write  
\[ A(t) = A_{\min}(t) + A_+(t) \]
where 
\[ A_{\min}(t) = \sum_{a=0}^{\en-1} 
\alpha_{a,\lambda_{\en-1-a}} t^{a\tu-\lambda_{\en-1-a}\en}. \]
Correspondingly 
we may write
\[ B(t) = \sum_{a=0}^{\tu-1} \sum_{b \in \hele} \beta_{a,b} t^{a\en-b\tu}. \]
For each $a$ let $\mu_{\tu-1-a}$ be the maximum of the set 
$\{ b \, | \, \beta_{a,b} \neq 0 \}$. We may then write  
\[ B(t) = B_{\min}(t) + B_+(t) \]
where 
\[ B_{\min}(t) = \sum_{a=0}^{\tu-1} \beta_{a,\mu_{\tu-1-a}} t^{a\en-\mu_{\tu-1-a}\en}. \]

\begin{proposition}\label{PosconePropABmin} Let $\en$ and $\tu$ be relatively prime.
Assume $A(t)$ and $B(t)$ are polynomials with nonnegative coefficients and
nonzero constant terms. Suppose
\[ A(t) \xi_\tu(t) = B(t) \xi_\en(t). \]
Let $\lambda$ and $\mu$ be the sequences corresponding to $A_{\min}(t)$ and
$B_{\min}(t)$. Then these sequences are partitions which are dual.
\end{proposition}

\begin{proof} Write the product above as $\sum \alpha_j t^j$.
Let $0 \leq b < \en-1$ and assume
$b\tu - \en\lambda$ occurs as a power in $A_{\min}(t)$, 
so $\lambda = \lambda_{\en-1-b}$.
We want to show that $\lambda_{\en-b-2} \geq \lambda_{\en-1-b}$.
If $(b+1)\tu - \en\lambda$ occurs as a power in $A(t)$ then clearly 
$\lambda_{\en-2-b} \geq \lambda = \lambda_{\en-1-b}$. 
So assume $(b+1)\tu - \en\lambda$ does not occur in $A(t)$.
By Lemma \ref{PosLemHopp} applied to $A(t) \xi_q(t)$:
\[ \alpha_{(b+1)\tu-\en\lambda} < \alpha_{(b+1)\tu-\en\lambda-1}\]
so $(b+1)\tu-\en(\lambda+1)$ is a power in $B(t)$.
We may now write
\[(b+1)\tu-\en(\lambda+1) = (\tu-\lambda-1)\en - \tu(\en-b-1). \]
There will then be an $a^\prime \leq \tu-\lambda-1$ such that 
$a^\prime \en - \tu(\en-b-1)$ is in $B(t)$ but
$(a^\prime -1) \en - \tu(\en-b-1)$ is not. 
By Lemma \ref{PosLemHopp} applied to $B(t) \xi_p(t)$: 
\[ \alpha_{a^\prime \en - \tu(\en-b-1)} > 
\alpha_{a^\prime \en - \tu(\en-b-1)-1}. \]
Now we may write
\[a^\prime \en - \tu(\en-b-1) = (b+1)\tu - \en(\tu-a^\prime)\]
and recall that  $\tu-a^\prime \geq \lambda + 1$.
Again by Lemma \ref{PosLemHopp} we get that the number in this equation 
will occur as a power
in $A(t)$. But this means that 
\[ \lambda_{\en-2-b} \geq \tu - a^\prime > \lambda 
= \lambda_{\en-1-b}. \]
Since $A(t)$ and equivalently $B(t)$ has nonzero constant term, we have
$\lambda_{\en-1} = 0$ so we get a partition $\lambda$.

An analog argument gives that the sequence of $\mu_i$'s also form a partition.

\medskip
Now let $T$ be the order ideal corresponding to $\lambda$ and $T^\prime$ the
order ideal corresponding to $\mu$. We show that they are equal and so
$\lambda$ and $\mu$ will be dual partitions.
Suppose $\lambda_{\en-b-1} < \lambda_{\en-b-2}$. 
Then $b\tu - \en \lambda_{p-b-2}$ is not in $A(t)$. By Lemma \ref{PosLemHopp}
\[ \alpha_{(b+1)\tu-\en\lambda_{\en-b-2}} > 
\alpha_{(b+1)\tu-\en \lambda_{\en-b-1}-1}.\]
And this implies again by Lemma \ref{PosLemHopp} that 
$(b+1)\tu-\en\lambda_{\en-b-2}$ occurs as a power in $B(t)$,
Rewriting, this is
$(\tu-1 - (\lambda_{\en-b-2}-1))\en - \tu(\en-b-1)$.
And this means that $(\lambda_{\en-b-2}-1, r)$ is in $T^\prime$ for some 
$r \geq \en-b-2$.
The upshot is that $T^\prime$ contains $T$. Analogously we could show
the opposite inclusion so these are in fact equal. 
\end{proof}

\begin{corollary} \label{PosconeCorABmin}
The polynomials $A(t) = \sum_{T,i} \gamma_{T,i} t^{c_{T,i}} A_T(t)$
and $B(t) = \sum_{T,i} \gamma_{T,i} t^{c_{T,i}} B_T(t)$
where the sum is over order ideals $T$ in $R(p,q)$ and a running index
$i$ for each $T$.
\end{corollary}

\begin{proof} Let  $\alpha$ be the minimal positive coefficient of 
$A_{\min}(t)$ and $B_{\min}(t)$ and suppose these correspond to the 
order ideal $T$. Then we can subtract off
$\alpha A_T(t)$ from $A(t)$ getting $A^\prime(t)$ and
similarly subtract off $\alpha B_T(t)$ and get $B^\prime(t)$.
Then also \[ A^\prime(t) \xi_\tu(t) = B^\prime(t) \xi_\en(t) \]
and we may proceed inductively, since then new polynomials have no more
terms than the original ones, and one of them strictly less.
\end{proof}

From this we obtain our goal of describing the extremal rays of the cone 
$P^\prime$ described at the end of Subsection \ref{SetSubsecLinto}.

\begin{theorem} \label{PosconeTheMain}
Let $e_1 = m\tu$ and $e_2 = m \en$ where $\en$ and $\tu$ are relatively
prime. The rays generated by $(t^cA_T(t^m), t^cB_T(t^m))$ where 
$T$ is an order ideal in 
$R(\en,\tu)$ and $c \in \hele$, are the extremal rays in the cone 
$P^\prime(e_1, e_2)$. 
In particular any element in this cone
may be written as a positive linear combination of these elements.
\end{theorem}

\begin{proof}
In the case $m=1$ this follows immediately from Proposition \ref{PosconeCorABmin}.
Suppose then $m > 1$. If we have 
\[  A(t) \xi_\tu(t^m) = B(t) \xi_\en(t^m) \]
we may write $A(t) = \sum_{i = 0}^{m-1} t^i A_i(t^m)$ and correspondingly for $B(t)$.
We must then have the equations
\[ A_i(t^m) \xi_\tu(t^m) = B_i(t^m) \xi_\en(t^m) \]
for each $i$. By Corollary  \ref{PosconeCorABmin} we may then conclude.
\end{proof}

\begin{remark} Such a positive linear combination is in general not unique.
\end{remark}

\begin{remark}
We see that the extremal rays fall into classes, one for each order ideal $T$ in 
$R(\en,\tu)$. These form a poset with a minimal element $T = \emptyset$ and a
maximal element $\hat T$. In Remarks \ref{PosRemMin} and \ref{PosRemMax} 
we showed that these correspond to Betti diagrams of well known resolutions.
\end{remark}


\section{Existence of resolutions} \label{EksisSec}

We will now show that for any extremal ray in $P^\prime (e_1,e_2)$ 
there is a resolution whose Betti diagram is in this extremal ray.
This will show that $P^\prime(e_1,e_2) = P(e_1,e_2)$. 

Given an order ideal $T$ in $R(p,q)$ where $p$ and $q$ are relatively prime.
If $e_1 = mp$ and $e_2 = mq$, we have the two
polynomials
$A_T(t^m)$ and $B_T(t^m)$. Homogenising these we may construct an
associated triple  $B_0,B_1,B_2$ fulfilling the equations of  
Lemma \ref{SetLemB}, with positive integer coefficients.
These lie on an extremal ray in $P^\prime(e_1,e_1)$.
Note that in $B_0$ and $B_2$ each monomial occurs with coefficient $0$ or $1$
and similarly for $B_2$.  We may therefore apply the following proposition whose 
proof will occupy this section.

\begin{proposition} \label{EksisPropTreB}
Let $(B_0,B_1,B_2)$ be a triple of homogeneous 
Laurent polynomials of increasing degrees, fulfilling the HK-equations
(\ref{SetLigHK}). If the coefficients of each monomial of $B_0$ and $B_2$
is $0$ or $1$, there is a resolution
\begin{equation} \label{EksiLigCx}
 S.B_0(t,u) \vmto{\alpha} S.B_1(t,u) \vmto{\beta} S.B_2(t,u)
\end{equation}
of an artinian $S$-module.
\end{proposition}

As a consequence we get the following.

\begin{theorem} \label{EksisTheMain}
Let $e_1 = m \tu$ and $e_2 = m \en$ where $p$
and $q$ are relatively prime. Let $(B_0, B_1, B_2)$ be the triple of 
homogeneous Laurent polynomials associated to an order ideal $T$ in $R(p,q)$, 
with $t^m$ as argument.
Then this is a triple of Betti polynomials associated to a bigraded artinian
module. Hence the cone $P(e_1, e_2) = P^\prime(e_1,e_2)$.
\end{theorem}

\begin{remark}
Proposition \ref{EksisPropTreB} holds for any $B_0$ and $B_2$ with 
nonnegative integer
coefficients. But for ease of demonstration we make the above assumptions.
\end{remark}

\begin{remark}
In the case of three variables it is not true that $P(e_1,e_2,e_2)$
is equal to $P^\prime(e_1,e_2,e_3)$. We provide an example where this
is not so in the last section.
\end{remark}

We shall prove the above proposition towards the end of this section.
But the following outlines what we need to show.
Since $\ker \alpha$ is a free module, 
$\ker \alpha/ \im \beta$ will be either $0$ or nonzero of codimension one
or zero.
But the latter is equivalent to $\coker \beta^{\vee}$ being of codimension 
one or zero. Hence we need to show the following.
\begin{itemize}
\item $\coker \alpha$ is of codimension two.
\item $\coker \beta^\vee$ is of codimension two.
\item The composition $\alpha \circ \beta = 0$.
\end{itemize}

First we have the following.

\begin{lemma} \label{EksisLemBigrad}
Given a bidegree $(i,j)$ with $i+j \geq \deg B_2(t,u)-1$. 
Then the dimension of the bigraded part $S.B_1(t,u)_{i,j}$ is the sum of the 
dimensions of $S.B_0(t,u)_{i,j}$ and $S.B_2(t,u)_{i,j}$.
\end{lemma}

\begin{proof}
The bigraded Hilbert function is
\[ h(t,u) = \frac{\sum_{i,\bfa} (-1)^i \beta_{i, \bfa} \cdot t^{a_1}t^{a_2}} 
{(1-t)(1-u)} \]
for some polynomial $h$. Writing $h(t,u)$ as $\sum \alpha_{i,j} t^iu^j$,
the coefficient $\alpha_{i,j}$ will be the alternating sum 
of the dimensions of the $S.B_p(t,u)_{i,j}$.
We will show that $\alpha_{i,j} = 0$ for $i+j \geq \deg B_2(t,u)-1$.
But if such a coefficient is nonzero, the pair $(i+1,j+1)$ must
occur as a power in the numerator in the fraction above. But this implies
in turn that $i+j+2$ is less or equal to the degree of $B_2(t,u)$.
\end{proof}

To facilitate the discussion we now introduce some notation. 
Let $s, e: [1, \ldots, n] \pil [1, \ldots, m]$ 
be two weakly increasing functions such that $s(i) \leq e(i)$. 
The subset $D = \{ (i,j) \, | \, s(i) \leq j \leq e(i) \}$ 
of  $[1, \ldots, n] \times [1, \ldots, m]$ is a {\it thick 
diagonal}. We then write $s = s_D$ and $e = e_D$.
If $s(1) = 1$ and $e(n) = m$ and $s$ and $e$ are strictly increasing 
we call $D$ a {\it strict} thick diagonal. If $s$ is only strictly increasing
as soon as $s(i) > 1$  and $e$ is only strictly increasing as long as 
$e(i) < m$, we call $D$ {\it semi-strict}.




Let $B_0(1,1) = p$ and $B_2(1,1) = q$ and write $B_0(t,u) = 
\sum_{i = 1}^{p} t^{a^1_i} u^{a^2_i}$ where $\{a^1_i\}$ is strictly
increasing and  $\{a^2_i\}$ is strictly decreasing.
Similarly for $B_2(t,u)$ with pairs $(c^1_k, c^2_k)$ 
and for $B_1(t,u)$ with pairs $(b^1_j, b^2_j)$ but now with
the $\{b^1_j\}$ only weakly increasing and the $\{b^2_j\}$ only weakly decreasing.

We may now note that the positions where $\alpha$ may have nonzero entries,
i.e. those pairs $(i,j)$ such that $(a^1_i, a^2_i) \leq (b^1_j, b^2_j)$, 
form a thick diagonal $D_\alpha$ of $[1, \ldots, p] \times [1, \ldots, p+q]$. 
It has no zero rows because of the HK-equations (\ref{SetLigHK}) : for
each $(a^1_i,a^2_i)$ there is a $(b^1_j, b^2_j)$ with $a^1_i = b^1_j$. Similarly
we have a thick diagonal $D_{\beta^\vee}$ in 
$[1,\ldots,q] \times [1, \ldots, p+q]$.

\begin{lemma} \label{EksisLemSe} 
a. $s_{D_\alpha}(i) = j$ if and only if $j$ is the smallest
index such that $a^1_i = b^1_j$. 

b. $e_{D_\alpha}(i) = j$ if and only if $j$ is the largest index for which
$a^2_i = b^2_j$.

\noindent The analog result holds for $D_{\beta^\vee}$.
\end{lemma}

\begin{proof} Let $s_{D_\alpha}(i) = j$ and let $\tilde{j}$ be the smallest index such
that  $a^1_i = b^1_{\tilde{j}}$. Such an index exists by the HK-equations.
Clearly $j \leq \tilde{j}$. But if $j < \tilde{j}$ then 
$b^1_{j} < b^1_{\tilde{j}}$ and so we could not have 
$(a^1_i, a^2_i) \leq (b^1_j, b^2_j)$
   The other arguments are analogous.
\end{proof}

\begin{corollary} \label{EksisCorStrict}
The thick diagonal $D_\alpha$ is strict. Similarly 
$D_{\beta^\vee}$ is strict.
\end{corollary}

\begin{proof} 
Since the $a^1_i$ are strictly increasing, we get that $s_{D_\alpha}$ is strictly increasing.
We thus need to show that $s_{D_\alpha}(1) = 1$. 
Let $s_{D_\alpha}(1) = j$. Suppose $j$ is not  $1$. Then $b^1_1 <  a^1_1$.
By the HK-equations there will be $(c^1_k, c^2_k)$
with $c^1_k = b^1_1$. But then again there will be a 
$(b^1_{j^\prime}, b^2_{j^\prime})$ with $b^2_{j^\prime} = c^2_k$ and this
would have $b^1_{j^\prime} < c^1_k = b^1_1$ which is impossible.
Thus  $s_{D_\alpha}(1) = j = 1$.
\end{proof}

\begin{lemma} \label{EksisLemRekt}
Let $D$ be a semi-strict diagonal of $[1, \ldots, n] \times [1, \ldots, n+1]$
with $e_D(1) > 1$ and $s_D(n) < n+1$.
Let $A$ be a general matrix of type $D$. Then there is a vector in the null
space of $A$ with nonzero first and last coordinates.
\end{lemma}

\begin{proof}
If we omit the first column we get an  $n \times n$-matrix of semi-strict
diagonal type. But a general such matrix is easily seen to be non-singular.
Hence a null vector must have nonzero first coordinate. Similarly for the last
coordinate.
\end{proof}

\begin{lemma} \label{EksisLemCoker} 
If  $\alpha$ is nonzero in positions $(i, s_{D_\alpha}(i))$
and $(i, e_{D_\alpha}(i))$ for $i = 1, \ldots, p$, 
then $\coker \alpha $ has codimension two.
Similarly for the map $\beta^\vee$.
\end{lemma}

\begin{proof}
By Lemma \ref{EksisLemSe}, 
in the first position of each row there is a power of $y$. Hence for
the matrix to degenerate we must have $y = 0$. Similarly there is a power
of $x$ in the last position, and so $x = 0$ when the matrix degenerates.
\end{proof}

Now when $\alpha$ and $\beta$ are composed, columns in $\beta$ are
multiplied with the rows of $\alpha$. Motivated by this we have the following.

\begin{lemma} \label{EksisLemAB} 
Let  $k$ be a column in $D_\beta$ which starts in position
$(j_0,k)$ and ends in $(j_1,k)$.  Then $D_\alpha$ restricted to 
$[1, \ldots, p] \times [j_0, j_1]$ has $j_1 - j_0$ nonzero rows,
say the interval $[i_0, i_1]$ where $j_1 - j_0 = i_1 - i_0 + 1$, and 
$D_\alpha$ restricted to $[i_0, i_1] \times [j_0, j_1]$ is semi-strict
with $e_{D_\alpha}(i_0) > j_0$ and $s_{D_\alpha}(i_1) < j_1$.
\end{lemma}

\begin{proof} 
\noindent 1. That $j_1 - j_0 = i_1 - i_0 +1$ follows from Lemma \ref{EksisLemBigrad}
by restricting to the bidegree $(c^1_k,c^2_k)$.

\noindent 2. Now we show $e_{D_\alpha}(i_0) > j_0$.
By the HK-equations there is a $(b^1_j, b^2_j)$ with
$b^2_j = c^2_k$. Since the $b^2_j$ are decreasing, this must happen for $j = j_0$.
(This is the analog of Lemma \ref{EksisLemSe} for $\beta^\vee$.)
Clearly $e_{D_\alpha}(i_0) \geq  j_0$. If we have equality,  
by Lemma \ref{EksisLemSe}
$a^2_{i_0} = b^2_{j_0}$. But then $a^2_{i_0} = c^2_k$ and by the HK-equations
there must then be two $(b^1_j, b^2_j)$ with $b^2_j = a^2_{i_0} = c^2_k$.
But this would again give $s_D(i_0) > j_0$. 
Similarly we can argue that $s_{D_\alpha}(i_1) < j_1$.

\noindent 3. That the restriction is semi-strict follows from 
i) $s_{D_\alpha}$ and $e_{D_\alpha}$ are strictly increasing, 
ii) $s_{D_\alpha}(i_o) \leq j_0$, and iii) $e_{D_\alpha}(i_1) \geq j_1$. 
To show ii) note that if $s_{D_\alpha}(i_0) > j_0$ then clearly
$s_{D_\alpha}(i_1) > j_0 + i_1 - i_0 = j_1 - 1$.
But this is not possible since $e_{D_\alpha} (i_1) \leq j_1 - 1$.
Similarly we can show iii).
\end{proof}

\begin{proof}[Proof of Proposition \ref{EksisPropTreB}.]

We choose $\alpha$ to be a general matrix, homogeneous with respect to the
multidegrees. It will be of type $D_\alpha$ and it degenerates in codimension
two by Lemma \ref{EksisLemCoker}.

 By Lemma \ref{EksisLemAB} 
we get for each column $k$ in $D_\beta$ a vector in the kernel
of $\alpha$ which is nonzero in positions  $s_{D_{\beta^\vee}}(k)$ and
$e_{D_{\beta^\vee}}(k)$. Hence these kernel vectors make up the columns of a
map $\beta$ such
that $\beta^\vee$ degenerates in codimension two by Lemma \ref{EksisLemCoker}. 
Also the composition $\alpha \circ \beta  = 0$, and this is what we needed to show.
\end{proof}

\section{Resolutions of trigraded artinian modules of codimension three}

In the case of trigraded artinian modules over the polynomial ring $\kr[x,y,z]$
where the resolution has pure total degrees,
we do not know much.
The following are natural questions.

\begin{itemize}
\item For Betti diagrams with given total degrees, are 
there, up to translation, only a finite number of 
extremal rays in the positive cone of such Betti diagrams?

\item Suppose the above property 2. holds. From Section \ref{PosSec} we know
that the translation classes of extremal rays form
a poset with a unique minimal member and a unique maximal member. Is there a 
maximal member in the translation classes in the three variable case
also?

\end{itemize}

We do not know the answer to these questions. 
A general fact we do know is that $L(\bfe) = L^\prime(\bfe)$. 
However in three variables
it is not the case that the injection $P(\bfe) \pil P^\prime(\bfe)$ is
an isomorphism. Let us consider as example the case of resolutions of type $0,1,2,1$. 
The equivariant resolution of this type has the form
(we have listed the tridegrees of the generators below each free module)
\begin{equation} \label{TrigLigETE} 
\underset{ \begin{matrix} 100 \\010 \\ 001 \end{matrix}}{S^3} \vpil 
\underset{ \begin{matrix} 200 \\020 \\002\\110 \\ 101 \\ 011 \end{matrix}}{S^6} \vpil 
\underset{ \begin{matrix} 211 \\121\\ 112\\220 \\ 202 \\ 022 \end{matrix}}{S^6} \vpil 
\underset{ \begin{matrix} 221 \\212 \\ 122 \end{matrix}}{S^3}. 
\end{equation}

To facilitate notation write
$\sum_i k_i\beta(a_i,b_i,c_i)$ as
$\sum_i [k_i(a_i,b_i,c_i)] \beta$.
Let $\beta$ be the Betti diagram of the complex (\ref{TrigLigETE}).
One may check that 
\[ [(2,1,0) + (0,2,1) + (1,0,2) - (1,1,1)] \beta \]
gives a diagram with no negative entries (and it fulfils the HK-equations).
But no multiple of this is the Betti diagram of a module. If $F_{\bullet}$
is a complex with this diagram, then $S(-3,-1,0)$ is a term in $F_0$. But 
there is no term $S(-3,-1,*)$ in $F_1$ (but there is one in $F_3$),
and so the cokernel of 
$F_1 \pil F_0$ cannot have codimension three.
In particular this diagram is in $P^\prime(1,2,1)$ but not in $P(1,2,1)$.
However let 
$\alpha$ be the diagram
\[ [(2,1,0) + (2,0,1) +  (1,2,0) +  (0,2,1) +  (1,0,2) + 
 (0,1,2) -  (1,1,1)] \beta.\]

\begin{claim}
The diagrams $\beta$ and $\alpha$ are Betti diagrams of resolutions of 
indecomposable
artinian trigraded modules of codimension three, and they generate rays which 
are extremal rays in the cone $P(1,2,1)$.
\end{claim}

\begin{proof}
That $\beta$ is a Betti diagram is clear and that it resolves an indecomposable module
is also immediate to see from the resolution. 
That $\alpha$ is a Betti diagram of a resolution of an indecomposable module,
may be checked on Macaulay 2 by filling in general monomial matrices with the tridegrees
of $\alpha$.
Now the only way $\alpha$ can decompose into nonnegative diagrams which
are not on its ray, may be worked out to be as follows.
\begin{eqnarray}
\notag &[ & c_1( (2,1,0) + (0,2,1) + (1,0,2) - (1,1,1)) \\
\notag  &+& c_2 ((1,2,0) + (0,1,2) +
(2,0,1) - (1,1,1)) + c_3 (1,1,1)] \beta 
\end{eqnarray}
where $c_1 = c_2 = c_3$.  But the same argument used to show that
the diagram corresponding to the first term is not a resolution may be used
to show that a linear combination as above is the diagram of a resolution only if
$c_1 = c_2 = c_3$ is a positive integer.
\end{proof}

It would be interesting to know if there are other extremal rays in the cone $P$
apart from the translates of $\alpha$ and $\beta$.


\begin{thebibliography} {99}

\bibitem{BS} M.Boij, J.S\"oderberg, {\em Graded Betti numbers of 
Cohen-Macaulay modules and the multiplicity conjecture},
Journal of the London Mathematical Society, {\bf 78} no.1, (2008), p.78-101.

\bibitem{EFW} D.Eisenbud, G.Fl\o ystad, J.Weyman, 
{\em The existence of pure free resolutions}, 
arXiv:0709.1529, to appear in Annales de l'institut Fourier.

\bibitem{ES} D.Eisenbud, F.-O. Schreyer, {\em Betti numbers of graded modules
and cohomology of vector bundles}, Journal of the American Mathematical 
Society {\bf 22} (2009),
p.859-888.

\bibitem{Fl} G.Fl\o ystad, {\em The linear space of Betti diagrams of 
multigraded artinian modules}, arXiv:1001.3235.

\bibitem{FuH} W.Fulton, J.Harris. {\em Representation theory},
GTM 129, Springer Verlag 1991.

\bibitem{MS} E.Miller, B.Sturmfels, {\em Combinatorial commutative algebra}
GTM 227,  Springer Verlag 2005.
\end{thebibliography}
\end{document}